\documentclass{amsart}

\usepackage[T1]{fontenc}
\usepackage{enumerate, amsmath, amsfonts, amssymb, amsthm, thmtools, dsfont, mathrsfs, wasysym, graphics, graphicx, xcolor, url, hyperref, hypcap, xargs, multicol, pdflscape, multirow, hvfloat, array, ae, aecompl, pifont, mathtools, a4wide, float, blkarray, overpic, nicefrac, stmaryrd, anyfontsize, yfonts}
\usepackage{CJKutf8}
\usepackage{xargs, bbm, enumerate, paralist}
\usepackage[shortlabels, inline]{enumitem}
\usepackage{pdflscape}
\usepackage[noabbrev,capitalise]{cleveref}
\usepackage[normalem]{ulem}
\usepackage{marginnote}
\usepackage{animate}

\hypersetup{colorlinks=true, citecolor=darkblue, linkcolor=darkblue}
\usepackage[all]{xy}
\usepackage{tikz}
\usepackage{tikz-cd}
\usetikzlibrary{trees, decorations, decorations.pathmorphing, decorations.markings, decorations.shapes, shapes, arrows, matrix, calc, fit, intersections, patterns, angles}
\graphicspath{{figures/}{figures/diagonals/}{figures/walks/}{figures/tubes/}{figures/blocks/}}
\makeatletter\def\input@path{{figures/}}\makeatother
\usepackage{caption}
\usepackage[export]{adjustbox}

\usepackage{paralist}
\usepackage{shuffle}

\newtheorem{theorem}{Theorem}[section]
\newtheorem{corollary}[theorem]{Corollary}
\newtheorem{proposition}[theorem]{Proposition}
\newtheorem{lemma}[theorem]{Lemma}
\newtheorem{conjecture}[theorem]{Conjecture}
\newtheorem*{theorem*}{Theorem}

\theoremstyle{definition}
\newtheorem{definition}[theorem]{Definition}

\newtheorem{remark}[theorem]{Remark}

\crefname{equation}{Equation}{Equations}

\newcommand{\R}{\mathbb{R}} 

\renewcommand{\c}[1]{{\mathcal{#1}}} 
\renewcommand{\b}[1]{{\boldsymbol{#1}}} 
\newcommand{\go}[1]{{\textgoth{#1}}} 

\renewcommand{\epsilon}{\varepsilon} 

\newcommand{\set}[2]{\left\{ #1 \mid #2 \right\}} 
\newcommand{\ssm}{\smallsetminus} 
\newcommand{\dotprod}[2]{\langle #1 \mid #2 \rangle} 
\newcommand{\one}{{1\!\!1}} 
\newcommandx{\ones}[1][1=n]{\one_{#1}} 
\newcommand{\eqdef}{\mbox{\,\raisebox{0.2ex}{\scriptsize\ensuremath{\mathrm:}}\ensuremath{=}\,}} 

\DeclareMathOperator{\conv}{conv} 

\newcommand{\ie}{\textit{i.e.}~} 
\definecolor{darkblue}{rgb}{0,0,0.7} 
\definecolor{green}{RGB}{57,181,74} 
\definecolor{violet}{RGB}{147,39,143} 
\newcommand{\darkblue}{\color{darkblue}} 
\newcommand{\defn}[1]{\textsl{\darkblue #1}} 

\makeatletter
\def\part{\@startsection{part}{1}%
\z@{.7\linespacing\@plus\linespacing}{.8\linespacing}%
{\LARGE\sffamily\centering}}
\makeatother

\newcommand{\polytope}[1]{\mathsf{#1}}

\newcommandx{\PivotPolytope}[2][1=d,2=\t]{\polytope{\Pi}_{#2}^{#1}}
\newcommandx{\PP}{\polytope{\Pi}}

\newcommandx{\HypSimpl}[2][1=n,2=k]{\polytope{\Delta}(#1,#2)}

\newcommandx{\MPP}[2][1=\polytope{P},2=\b c]{\polytope{M}_{#2}(#1)}
\newcommandx{\MPPHypSimpl}[2][1=n,2=k]{\polytope{M}(#1,#2)}

\newcommandx{\gZono}[1][1=G]{\mathsf{Z}_{#1}}

\newcommandx{\Asso}[2][1=m,2={}]{\mathsf{Asso}^{#2}(#1)} 
\newcommandx{\dZono}[1][1=\b{h}]{\mathsf{D}_{#1}} 
\newcommand{\simplex}{\polytope{\Delta}} 

\newcommandx{\Fan}[1][1=F]{\mathcal{#1}} 
\newcommandx{\nestedFan}[1][1=\quiver]{\mathcal{F}(#1)} 

\newcommandx{\ray}[1][1=r]{\b{#1}} 
\newcommandx{\rays}[1][1=R]{\b{#1}} 
\newcommandx{\Perm}[1][1=n]{\polytope{Perm}_{#1}}

\newcommandx{\gArr}[1][1=G]{\mathcal{A}_{#1}} 
\newcommandx{\gFan}[1][1=G]{\Fam_{#1}} 
\newcommandx{\gFanO}[1][1=G]{\widehat{\Fan}_{#1}} 
\newcommandx{\cc}[1][1=G]{\mathbb{K}_{#1}} 

\newcommandx{\braid}[1][1=n]{\mathcal{B}_{#1}} 
\newcommandx{\sbraid}[1][1=n]{\widehat{\mathcal{B}}_{#1}} 

\newcommandx{\coefficient}[3][1={\b{s}}, 2=\b{r}, 3=\b{r}']{\alpha_{#2,#3}(#1)} 
\newcommandx{\virtualPolytopes}[1][1=d]{\mathbb{V}^{#1}} 
\newcommandx{\VDP}[1][1=n]{\mathbb{VDP}^{#1}} 
\newcommandx{\CVDP}[1][1=n]{\overrightarrow{\mathbb{VDP}}^{#1}} 
\newcommand{\VD}[1][1=n]{\mathbb{VD}} 
\newcommandx{\opcone}[1][1={\mu,\omega}]{\polytope{C}_{#1}}
\newcommandx{\orcone}[1][1={\omega}]{\polytope{C}_{#1}}

\newcommandx{\graphG}[1][1=G]{#1} 
\newcommandx{\hypergraph}[1][1=H]{\graphG[#1]} 
\newcommandx{\tube}[1][1=t]{\mathsf{#1}} 
\newcommandx{\tubes}[1][1=\graphG]{\building#1} 
\newcommandx{\tubing}[1][1=T]{\mathsf{#1}} 

\newcommand{\building}{\mathcal{B}} 
\newcommandx{\nested}[1][1=N]{\mathcal{#1}} 

\newcommandx{\enhancedStep}[3][1=i, 2=j, 3=a]{#1 \xrightarrow{#3} #2}
\newcommandx{\step}[2][1=i, 2=j]{#1 \rightarrow #2}
\newcommandx{\enhancedStepx}[3][1=x, 2=y, 3=z]{#1 \xrightarrow{#3} #2}
\newcommandx{\enhancedStepZ}[3][1=x, 2=y, 3=Z]{#1 \xrightarrow{#3} #2}

\renewcommand\t{\b t}%

\newcommandx{\ConstrainedMALIntrinsic}[2][1=n,2=d]{\go M_{#1,#2}}

\newcommandx{\AssGraph}[2][1=n,2=3]{\mathit{Asso}_{#1}^{#2}}

\newcommandx{\Fpolytope}[3][1=d,2=A,3=\b t]{\polytope{P}_{#1}^f\left(#2,#3 \right)}
\newcommandx{\Bpolytope}[3][1=d,2=A,3=\b t]{\polytope{P}_{#1}^b \left(#2,#3 \right)}

\newcommandx{\FibPol}[3][1=\polytope{P},2=\polytope{Q},3=\pi]{\polytope{\Sigma}_{#3}(#1,#2)}

\newcommandx{\FibPolCyc}[2][1=d,2=\t]{\polytope{\Sigma}^{#1}_{2}(#2)}
\newcommandx{\PolProj}[3][1=\polytope{P},2=\polytope{Q},3=\pi]{#3~:~#1\to #2}

\newcommandx{\HOmega}[3][1=\kappa,2=\t,3=d]{\Omega_{\kappa}^d(\t)}
\newcommandx{\rHOmega}[3][1=\kappa,2=\t,3=d]{\overline{\Omega}_{\kappa}^d(\t)}
\newcommandx{\Ppolytope}[3][1=d,2=T,3=\t]{\polytope{Q}^+_{#1}(#2,#3)}
\newcommandx{\Npolytope}[3][1=d,2=T,3=\t]{\polytope{Q}^-_{#1}(#2,#3)}

\newcommandx{\gVdM}[3][1=n,2=k,3=\b \lambda]{\text{VdM}_{#1,#2}(#3)}
\newcommandx{\VdM}[2][1=n,2=\b\lambda]{\text{VdM}_{#1}(#2)}

\newcommandx{\PivotFan}[2][1=\polytope{P},2=\b c]{\c P_{#1, #2}}
\newcommand{\SlopeMap}{\theta}
\newcommand{\LodayFan}[1][m]{\c L_{#1}}
\newcommand{\ProductSlopeMap}{\Theta}

\usepackage{todonotes}

\title[Pivot polytopes of products of simplices and shuffles of associahedra]{Pivot polytopes of products of simplices \\ and shuffles of associahedra}

\thanks{
VP was partially supported by the French project CHARMS (ANR~19\,CE40\,0017), by the French\,--\,Austrian project PAGCAP (ANR~21\,CE48\,0020 \& FWF I 5788), by the Spanish projects PID2019-106188GB-I00 and PID2022-137283NB-C21 of MCIN/AEI/10.13039/501100011033, by the Severo Ochoa and María de Maeztu Program for Centers and Units of Excellence in R\&D (CEX2020-001084-M), and by the Departament de Recerca i Universitats de la Generalitat de Catalunya (2021 SGR 00697).
}

\author{Vincent Pilaud}
\address[V.~Pilaud]{Universitat de Barcelona \& Centre de Recerca Matemàtica, Barcelona}
\email{vincent.pilaud@ub.edu}
\urladdr{\url{https://www.ub.edu/comb/vincentpilaud/}}

\author{Germain Poullot}
\address[G.~Poullot]{Universität Osnabrück, Germany}
\email{germain.poullot@uni-osnabrueck.de}

\begin{document}

\begin{abstract}
We provide a piecewise linear isomorphism from the normal fan of the pivot polytope of a product of simplices to the normal fan of a shuffle of associahedra.
\end{abstract}

\vspace*{-2cm}

\maketitle

\section{Introduction}

To solve a linear program given by a polytope~$\polytope{P}$ and a direction~$\b c$, the simplex algorithm traverses a path along the graph of~$\polytope{P}$ from any given vertex to a maximal vertex, choosing for each vertex an improving neighbor according to a given pivot rule.
The pivot rule is memoryless when the choice only depends on the current vertex, which can be encoded by an arborescence mapping each vertex to its preferred neighbor.
The classical shadow-vertex pivot rule, instrumental for randomized and smoothed analysis of the simplex method~\cite{Borgwardt,SpielmanTeng}, is not memoryless.
To make the shadow-vertex pivot rule memoryless,
A.~Black, J.~De Loera, N.~L\"utjeharms and R.~Sanyal defined in~\cite{BlackDeLoeraLutjeharmsSanyal} the max-slope pivot rule with respect to a given fixed generic weight~$\b\omega$, which chooses the improving neighbor maximizing the slope on the plane defined by~$\b c$ and~$\b\omega$.
They also introduced the max-slope pivot rule polytope (that we abbreviate here by pivot polytope), whose vertices are in bijection to the arborescences of the max-slope pivot rule on~$(\polytope{P}, \b c)$.
They observed that the pivot polytope of a cube is the standard permutahedron, that the pivot polytope of a simplex is an associahedron, and that the pivot polytope of a prism over a simplex is a multiplihedron (these observations were latter proved in~\cite{BlackLutjeharmsSanyal}).
\cref{fig:examples} illustrates these miracles in dimension $2$.
Based on enumerative data, V.~Pilaud and R.~Sanyal further conjectured~\cite{PilaudSanyal} that the pivot polytope of a product of two simplices is a constrainahedron~\cite{BottmanPoliakova}.
As multiplihedra and constrainahedra are both obtained from associahedra by the shuffle operation of F.~Chapoton and V.~Pilaud~\cite{ChapotonPilaud-shuffle}, it naturally led to the following conjecture.

\begin{figure}[t]
	\centerline{
		\includegraphics[scale=.7]{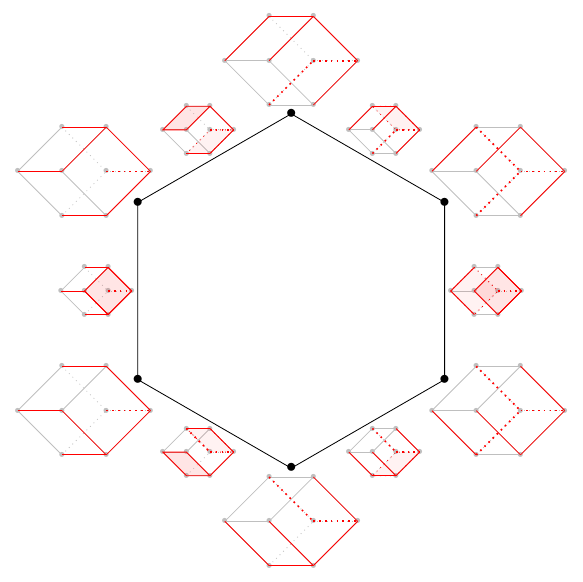}
		\quad
		\raisebox{.3cm}{\includegraphics[scale=.7]{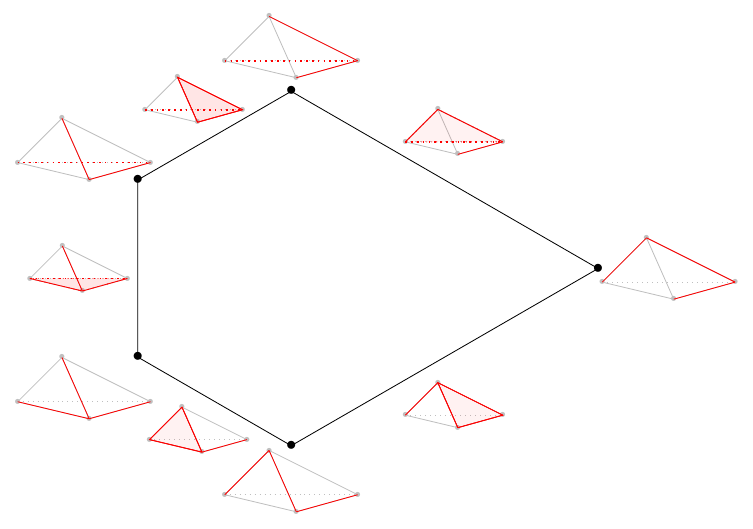}}
	}
	\caption{The arborescences of the max-slope pivot rule over the $3$-dimensional cube (left) and simplex (right) correspond to the faces of the $2$-dimensional permutahedron (left) and associahedron (right). The objective vector~$\b c$ points to the right.}
	\label{fig:examples}
\end{figure}

\begin{conjecture}[\cite{PilaudSanyal}]
\label{conj:main}
For any~$m_1 \ge 1, \dots, m_t \ge 1$, the pivot polytope of the product of simpli\-ces~${\simplex_{m_1} \times \dots \times \simplex_{m_t}}$ and the shuffle of Loday's associahedra~$\Asso[m_1] \star \dots \star \Asso[m_t]$ are combinatorially isomorphic (meaning that they have isomorphic face lattices).
\end{conjecture}

Partial cases of this conjecture were solved in~\cite{BlackDeLoeraLutjeharmsSanyal,BlackLutjeharmsSanyal}, namely when~${m_1 \! = \! \dots \! = \! m_t \! = \! 1}$ (permutahedron~\cite[Thm.~6.5]{BlackDeLoeraLutjeharmsSanyal}), when~$t = 1$ (associahedron~\cite[Thm.~4.3]{BlackLutjeharmsSanyal}), when~${t = 2}$ (constrainahedron~\cite[Thm.~5.8]{BlackLutjeharmsSanyal}), and the vertex count when~$m_1 = \dots = m_{t-1} = 1$ (multiplihedron~\cite[Thm.~6.2]{BlackLutjeharmsSanyal}).
These results were achieved using the connection to particle collisions~\cite{BottmanPoliakova}, motivated by the case of constrainahedra.

This paper reports on an alternative approach to pivot polytopes of product of simplices, developed independently from~\cite{BlackLutjeharmsSanyal} and originally announced in~\cite{Poullot}.
This approach closes \cref{conj:main} and provides arguably simpler proofs even for the known cases of the associahedron and constrainahedron.
Our main idea is to define the slope map, which sends each weight~$\b\omega$ to its slope vector, recording the optimal slope at each vertex.
As the graph of the simplex is complete, the arborescence for~$\b\omega$ can be directly retrieved from the order of these optimal slopes, hence from the region of the braid arrangement containing the slope vector of~$\b\omega$.
Studying the regions giving the same arborescence naturally leads to our first result, which refines \cite[Thm.~4.3]{BlackLutjeharmsSanyal}.

\begin{theorem}
\label{thm:main1}
For any full dimensional simplex~$\simplex \subset \R^m$ and any generic direction~$\b c \in \R^m$, the slope map is a piecewise linear isomorphism from the normal fan of the pivot polytope of~$(\simplex, \b c)$ to the normal fan of the associahedron~$\Asso[m]$.
\end{theorem}

Beyond the case of simplices for which it provides a proof of \cref{thm:main1}, the slope map connects pivot polytopes to deformed permutahedra (or generalized permutahedra~\cite{Postnikov}).
More precisely, it always embeds the pivot fan inside the braid fan.
This perspective opens several research directions discussed in~\cite{Poullot}.

We then exploit the fact that a product of polytopes contains many parallel edges to define a product slope map, which sends each weight~$\b\omega$ to a vector recording irredundantly the optimal slope at each vertex.
This map enables us to prove the following refinement of \cref{conj:main}.

\begin{theorem}
\label{thm:main2}
For any full dimensional simplices~$\simplex_1 \subset \R^{m_1}, \dots, \simplex_t \subset \R^{m_t}$ and any generic direction~$\b c \in \R^{m_1 + \dots + m_t}$, the product slope map is a piecewise linear isomorphism from the normal fan of the pivot polytope of~${(\simplex_1 \times \dots \times \simplex_t, \b c)}$ to the normal fan of the shuffle of associahedra~${\Asso[m_1] \star \dots \star \Asso[m_t]}$.
\end{theorem}

We insist that both \cref{thm:main1,thm:main2} deal with piecewise linear maps on fans.
It implies that the polytopes are combinatorially equivalent, but not necessarily normally equivalent.
In fact they are not, and as observed in~\cite{BlackLutjeharmsSanyal}, the pivot polytopes of simplices seem to be new geometric realizations of the associahedron.

The paper is organized as follows.
In \cref{sec:pivotFan}, we recall from~\cite{BlackDeLoeraLutjeharmsSanyal} the constructions of the pivot fan and pivot polytope.
In \cref{sec:simplex}, we define the slope map and prove \cref{thm:main1} for the simplex as a warm up for the general case.
Finally, we prove \cref{thm:main2} in \cref{sec:productSimplices}.

\section{Pivot fan and pivot polytope}
\label{sec:pivotFan}

A \defn{linear program} is a pair $(\polytope{P}, \b c)$ where $\polytope{P}\subset\R^d$ is a $d$-dimensional polytope and $\b c\in \R^d$ is the direction to be optimized.
We denote by~$V(\polytope{P})$ and~$E(\polytope{P})$ the vertex and edge sets of~$\polytope{P}$, and let~$n \eqdef |V(\polytope{P})|$ and~$m \eqdef n-1$.
We assume that $(\polytope{P}, \b c)$ is \defn{generic} in the sense that~${\dotprod{\b c}{\b u} \ne \dotprod{\b c}{\b v}}$ for any~$\b u \b v \in E(\polytope{P})$, where~$\dotprod{\cdot}{\cdot}$ denotes the standard scalar product of~$\R^d$.
An \defn{improving neighbor} of~$\b u \in V(\polytope{P})$ is any~$\b v \in V(\polytope{P})$ such that~$\b{uv} \in E(\polytope{P})$ and~$\dotprod{\b c}{\b u} < \dotprod{\b c}{\b v}$.
By genericity, there is a unique~$\b v_{\max} \in V(\polytope{P})$ maximizing~$\dotprod{\b c}{\b v}$ for~$\b v \in V(P)$ (and it has no improving neighbor).
\mbox{For any~$\b u \ne \b v \in V(\polytope{P})$ and~$\b \omega\in \R^d$, we define:}
\[
\rho^{\b \omega}(\b u, \b v) \eqdef \frac{\dotprod{\b \omega}{\b v - \b u}}{\dotprod{\b c}{\b v - \b u}}.
\]

\begin{definition}[{\cite{BlackDeLoeraLutjeharmsSanyal}}]
\label{def:arborescence}
For a secondary direction $\b\omega\in\R^d$ linearly independent of $\b c$, we define
\[
\tau^{\b\omega} (\b u) \eqdef \max\set{\rho^{\b\omega}(\b u, \b v)}{\b v \text{ improving neighbor of } \b u}.
\]
We say that the direction~$\b\omega \in \R^d$ is \defn{generic} when there is a unique improving neighbor~$\b v$ of~$\b u$ with~${\tau^{\b\omega} (\b u) = \rho^{\b\omega}(\b u, \b v)}$, and we then define~$\c A^{\b\omega}(\b u) \eqdef \b v$. 
By convention we set~$\tau^{\b\omega}(\b v_{\max}) \eqdef -\infty$ and~$\c A^{\b\omega}(\b v_{\max}) \eqdef \b v_{\max}$.
The map~$\c A^{\b\omega} : V(\polytope{P}) \to V(\polytope{P})$ is called the \defn{arborescence} of~$\b\omega$.
\end{definition}

\begin{theorem}[{\cite{BlackDeLoeraLutjeharmsSanyal}}]
The closures of the fibers of the map $\b\omega\mapsto\c A^{\b\omega}$ are the maximal cones of a polyhedral fan~$\PivotFan$, called the \defn{pivot fan} of~$(\polytope{P}, \b c)$.
In other words, $\b\omega$ and $\b \omega'$ belong to the relative interior of the same maximal cone of $\PivotFan$ if and only if $\c A^{\b\omega} = \c A^{\b\omega'}$.
\end{theorem}

\begin{theorem}[{\cite[Thm.~5.4]{BlackDeLoeraLutjeharmsSanyal}}]
The pivot fan~$\PivotFan$ is the normal fan of a polytope, called the \defn{pivot polytope} of~$(\polytope{P}, \b c)$.
\end{theorem}

We skip the definition of the pivot polytope as we work at the level of the pivot fan.
Our aim is to construct a piecewise linear map that embeds the pivot fan into the braid fan, especially for products of simplices.
From \cref{def:arborescence}, it is natural to consider the function~$\b\omega \mapsto \tau^{\b\omega}(\b u)$ (which also appeared in the proof of~\cite[Thm.~1.4]{BlackDeLoeraLutjeharmsSanyal} as the support function of the pivot~polytope).

\begin{lemma}
\label{lem:piecewiseLinear}
For any~$\b u \in V(\polytope{P})$, the map~$\b\omega \mapsto \tau^{\b\omega}(\b u)$ is piecewise linear on the cones of the pivot fan~$\PivotFan$.
\end{lemma}

\begin{proof}
For any $\b v \in V(\polytope{P})$, the map $\b\omega \mapsto \rho^{\b\omega}(\b u, \b v)$ is linear (from $\R^d$ to~$\R$).
If $\b \omega$ and $\b \omega'$ belong to the interior of the same maximal cone of the pivot fan~$\PivotFan$, then $\tau^{\b\omega}(\b u) = \rho^{\b\omega}(\b u, \b v)$ and $\tau^{\b\omega'}(\b u) = \rho^{\b\omega'}(\b u, \b v)$ for the same $\b v = \c A^{\b\omega}(\b u) = \c A^{\b\omega'}(\b u)$.
Hence, $\b\omega \mapsto \tau^{\b\omega}(\b u)$ is linear on the interior of any maximal cone of the pivot fan~$\PivotFan$.
As $\b\omega \mapsto \tau^{\b\omega}(\b u)$ is continuous, it is piecewise linear on the closed cones of the pivot fan~$\PivotFan$.
\end{proof}

\section{Pivot fan of a simplex}
\label{sec:simplex}

\enlargethispage{.1cm}
In this section, we show that the pivot polytope of a simplex~$\simplex_m$ is combinatorially equivalent to the associahedron~$\Asso[m]$.
More precisely, we provide an explicit piecewise linear map from the pivot fan of~$\simplex_m$ to the normal fan of~$\Asso[m]$.

\subsection{Sylvester fan and associahedron}

First, we just briefly recall that the normal fan of Loday's associahedron~$\Asso[m]$ is obtained by coarsening the braid arrangement according to the sylvester congruence.
We refer to \cite{PilaudSantosZiegler} for a detailed survey.

\begin{definition}
The \defn{braid arrangement} is the arrangement of the hyperplanes~$\set{\b{x} \in \R^m}{x_i = x_j}$ for all~$1 \le i < j \le m$.
It has a region~$C(\pi) \eqdef \set{\b{x} \in \R^m}{x_{\pi_1} < \dots < x_{\pi_m}}$ for each permutation~$\pi$ of~$[m]$.
Two regions~$C(\pi)$ and~$C(\pi')$ are adjacent if~$\pi$ and~$\pi'$ are adjacent permutations, meaning that~$\pi = XijY$ and~$\pi' = XjiY$ for two letters~$i,j \in [m]$ and two words~$X,Y$ on~$[m]$.
\end{definition}

\begin{definition}
\label{def:sylvesterCongruence}
The \defn{sylvester congruence} is the equivalence relation~$\equiv_\mathrm{sylv}$ on permutations of~$[m]$ defined by the transitive closure of the rewriting rule~$UjVikW \equiv_\mathrm{sylv} UjVkiW$ for some letters~${1 \le i < j < k \le m}$ and some words~$U,V,W$ on~$[m]$.
In other words, two permutations~$XikY$ and~$XkiY$ adjacent by a simple transposition are sylvester congruent if and only if there exists a letter~$j \in X$ such that~$i < j < k$.
\end{definition}

\begin{definition}
\label{def:sylvesterFan}
The \defn{sylvester fan}~$\LodayFan[m]$ is the fan of~$\R^m$ whose maximal cones are the unions of the regions~$C(\pi)$ of the braid arrangement corresponding to permutations~$\pi$ of~$[m]$ in the same class of the sylvester congruence.
\end{definition}

\begin{proposition}[\cite{Loday,Postnikov}]
\label{prop:Loday}
The sylvester fan is the normal fan of the \defn{associahedron}
\[
\Asso \eqdef \sum_{1 \le i < k \le m} \conv\set{-\b{e}_j}{i \le j \le k}.
\]
\end{proposition}

\begin{remark}
The attentive reader has noticed our unusual conventions in \cref{def:sylvesterCongruence} and \cref{prop:Loday}.
The sylvester congruence is usually the transitive closure of~${UikVjW \equiv UkiVjW}$, and defines the normal fan of Loday's associahedron~$\sum_{1 \le i < k \le m} \conv\set{\b{e}_j}{i \le j \le k}$.
Our definitions, sometimes called anti-sylvester congruence and anti-associahedron, are equivalent up to central symmetry and fit better the max-slope pivot rule.
\end{remark}

\subsection{Slope map}
\label{subsec:slopeMap}

We now assume that the graph of~$\polytope{P}$ is complete, and we denote by $\b u_1, \dots, \b u_n$ the vertices of~$\polytope{P}$ such that~$\dotprod{\b c}{\b u_i} < \dotprod{\b c}{\b u_j}$ for~$i < j$.
We therefore abuse notation, considering that~$\rho^{\b\omega} : [n] \times [n] \to \R$, that~$\tau^{\b\omega} : [n] \to \R\cup\{-\infty\}$, and that~$\c A^{\b\omega} : [n] \to [n]$ for any generic~$\b\omega$.
The following statement builds on~\cite[Lem.~3.6]{BlackLutjeharmsSanyal}.

\begin{lemma}
\label{lem:formulaACompleteGraph1}
If the graph of $\polytope{P}$ is complete, then for any generic~$\b\omega \in \R^d$ and~$i \in [m]$, we have
\[
\c A^{\b\omega}(i) = \min \set{j \in [n]}{i < j \text{ and } \tau^{\b\omega}(i) > \tau^{\b\omega}(j)}.
\]
\end{lemma}

\begin{proof}
For $i < j < k$, we have
\[
\rho^{\b\omega}(i,k) = \frac{\dotprod{\b c}{\b u_j - \b u_i}}{\dotprod{\b c}{\b u_k - \b u_i}}\rho^{\b\omega}(i,j) + \frac{\dotprod{\b c}{\b u_k - \b u_j}}{\dotprod{\b c}{\b u_k - \b u_i}}\rho^{\b\omega}(j,k)
\]
which is a strict convex combination, so that~$\rho^{\b\omega}(i,k)$ separates $\rho^{\b\omega}(i,j)$ from $\rho^{\b\omega}(j,k)$.

Fix a generic $\b\omega \in \R^d$.
For $i \in [m]$ we have~$\c A^{\b\omega}(i) > i$ by definition.
For any $j$ with~${i < j < \c A^{\b\omega}(i)}$, we know that~$j$ is an improving neighbor of~$i$ (because $i < j$ and the graph of $\polytope{P}$ is complete), hence $\rho^{\b\omega}(i, j) < \rho^{\b\omega}(i, \c A^{\b\omega}(i))$.
Applying the separation argument to~$i < j < \c A^{\b\omega}(i)$, we obtain  ${\rho^{\b\omega}(i, j) < \rho^{\b\omega}(i, \c A^{\b\omega}(i)) < \rho^{\b\omega}(j, \c A^{\b\omega}(i))}$.
Hence,~${\tau^{\b\omega}(i) = \rho^{\b\omega}(i, \c A^{\b\omega}(i)) < \rho^{\b\omega}(j, \c A^{\b\omega}(i)) \leq \tau^{\b\omega}(j)}$. \linebreak
Finally, observe that~$\rho^{\b\omega}\bigl(i, \c A^{\b\omega}(\c A^{\b\omega}(i))\bigr) < \rho^{\b\omega}(i, \c A^{\b\omega}(i))$.
Applying again the separation argument to~${i < \c A^{\b\omega}(i) < \c A^{\b\omega}(\c A^{\b\omega}(i))}$, we obtain ${\rho^{\b\omega}(\c A^{\b\omega}(i), \c A^{\b\omega}(\c A^{\b\omega}(i))) < \rho^{\b\omega}(i, \c A^{\b\omega}(\c A^{\b\omega}(i))) < \rho^{\b\omega}(i, \c A^{\b\omega}(i))}$.
Hence~$\tau^{\b\omega}(\c A^{\b\omega}(i)) = \rho^{\b\omega}(\c A^{\b\omega}(i), \c A^{\b\omega}(\c A^{\b\omega}(i))) < \rho^{\b\omega}(i, \c A^{\b\omega}(i)) = \tau^{\b\omega}(i)$.
This proves the lemma.
\end{proof}

For a generic~$\b\omega \in \R^d$, we define~\defn{$\pi^{\b\omega}$} as the permutation of~$[m]$ such that $\tau^{\b\omega}(\pi^{\b\omega}_1) < \dots < \tau^{\b\omega}(\pi^{\b\omega}_m)$.

\begin{lemma}
\label{lem:sylvester1}
If the graph of $\polytope{P}$ is complete, then for any two generic~$\b\omega, \b\omega' \in \R^d$, we have~$\c A^{\b\omega} = \c A^{\b\omega'}$ if and only if $\pi^{\b\omega} \equiv_\mathrm{sylv} \pi^{\b\omega'}$.
\end{lemma}

\begin{proof}
We assume that~$\pi^{\b\omega}$ and~$\pi^{\b\omega'}$ are adjacent permutations; the general case follows by induction on the distance from~$\pi^{\b\omega}$ to~$\pi^{\b\omega'}$ (\ie the minimum number of simple transpositions necessary to transform~$\pi^{\b\omega}$ to~$\pi^{\b\omega'}$).
Hence~$\pi^{\b\omega} = XikY$ while~${\pi^{\b\omega'} = XkiY}$ for some letters~${1 \le i < k \le m}$ and some words~$X,Y$ on~$[m]$.
For any~$p \in [n]$, we have
\(
\set{j \in [n]}{p < j \text{ and } \tau^{\b\omega}(p) > \tau^{\b\omega}(j)} = \{\pi^{\b\omega}_1, \pi^{\b\omega}_2, \dots, \pi^{\b\omega}_q\} \ssm [p]
\)
where~$p = \pi^{\b\omega}_q$.
Note that~$\{\pi^{\b\omega}_1, \pi^{\b\omega}_2, \dots, \pi^{\b\omega}_q\} \ssm [p] = \{\pi^{\b\omega'}_1, \pi^{\b\omega}_2, \dots, \pi^{\b\omega'}_q\} \ssm [p]$ except if~$p = i$.
Consequently, \cref{lem:formulaACompleteGraph1} ensures that~${\c A^{\b\omega}(p) \!=\! \c A^{\b\omega'}(p)}$ except maybe when~$p = i$.
We distinguish two cases:
\begin{itemize}
    \item if there is~$j \in X$ with~$i < j < k$, then~$\c A^{\b\omega}(i) = \c A^{\b\omega'}(i) \le j$,
    \item otherwise, $\c A^{\b\omega}(i) \ne k = \c A^{\b\omega'}(i)$.
    \qedhere
\end{itemize}
\end{proof}

\cref{lem:sylvester1} motivates the following definition.

\begin{definition}
We define the \defn{slope map} $\SlopeMap: \R^d \to \R^m$ by~$\SlopeMap(\b \omega) = \bigl(\tau^{\b\omega}(i) \bigr)_{i \in [m]}$.
\end{definition}

\begin{lemma}
\label{lem:injective1}
On each maximal cone of the pivot fan~$\PivotFan$, the slope map~$\SlopeMap$ is linear and injective.
\end{lemma}

\begin{proof}
The map~$\SlopeMap$ is linear on each maximal cone of the pivot fan~$\PivotFan$ by~\cref{lem:piecewiseLinear}.
If~$\b\omega, \b\omega' \in \R^d$ are such that~$\c A^{\b\omega} = \c A^{\b\omega'}$ and~$\SlopeMap(\b\omega) = \SlopeMap(\b\omega')$, then all edges~$\b u_i \b u_{\c A^{\b\omega}(i)}$ are orthogonal to~$\b\omega-\b\omega'$.
As~$(\polytope{P}, \b c)$ is generic, the arborescence~$\c A^{\b\omega}$ is a spanning tree of the graph of~$\polytope{P}$.
Hence, as~$\polytope{P}$ is full dimensional, so is the span of~$\set{\b u_i \b u_{\c A^{\b\omega}(i)}}{i \in [m]}$, which implies that $\b\omega = \b\omega'$.
\end{proof}

\subsection{Pivot fan of a simplex}

We are now ready to show that the slope map sends the pivot fan of the simplex~$\simplex_m$ to the sylvester fan~$\LodayFan$, which refines~\cite[Thm.~4.3]{BlackLutjeharmsSanyal}.

\begin{theorem}
\label{thm:PivotFanSimplex}
For any full dimensional simplex~$\simplex \subset \R^m$ and any generic direction $\b c \in \R^m$, the slope map~$\SlopeMap$ is a piecewise linear isomorphism from the pivot fan $\PivotFan[\simplex]$ to the sylvester fan~$\LodayFan$.
\end{theorem}

\begin{proof}
Note first that $\PivotFan[\simplex]$ and $\LodayFan$ are both complete fans in $\R^m$.
By \cref{lem:injective1}, the map $\SlopeMap$ sends each $m$-dimensional cone of~$\PivotFan[\simplex]$ to an $m$-dimensional cone in~$\R^m$, that we call $\SlopeMap$-cone in this proof.
By~\cref{lem:sylvester1}, each $\SlopeMap$-cone is contained in some maximal cone of $\LodayFan$, and distinct $\SlopeMap$-cones are contained in distinct maximal cones of~$\LodayFan$.
In particular, the interior of the $\SlopeMap$-cones are disjoint.
By continuity of~$\SlopeMap$, two adjacent maximal cones of~$\PivotFan[\simplex]$ are thus sent to adjacent $\SlopeMap$-cones in~$\R^m$.
We thus obtain that the $\SlopeMap$-cones form a complete fan in~$\R^m$.
As each $\SlopeMap$-cone is contained in a cone of~$\LodayFan$, and both are complete fans in $\R^m$, we conclude that they coincide.
\end{proof}

\begin{remark}
\label{rem:bijection}
Recall that the sylvester fan has a cone~$C(T) \eqdef \! \set{x \in \R^m}{x_i \le x_j \text{ for all } i \!\to\! j \text{ in } T}$ for each binary tree~$T$ with $m$ internal nodes (where~$T$ is labeled in in-order and oriented away from its root).
The slope map thus induces a bijection from the arborescences~$\set{\c A^{\b\omega}}{\b\omega \in \R^m}$ on~$(\simplex, \b c)$ to the binary trees with $m$ internal nodes.
This is the classical bijection from non-crossing arborescences to binary trees~\cite[Sect.~2, Item~21]{Stanley-CatalanNumbers}.
\end{remark}

\section{Pivot fan of a product of simplices}
\label{sec:productSimplices}

In this section, we prove that the pivot polytope of a product of simplices~$\simplex_{m_1} \times \dots \times \simplex_{m_t}$ is combinatorially equivalent to the shuffle of associahedra~$\Asso[m_1] \star \dots \star \Asso[m_t]$.
Again, we provide an explicit piecewise linear map from the pivot fan of~$\simplex_{m_1} \times \dots \times \simplex_{m_t}$ to the normal fan of~${\Asso[m_1] \star \dots \star \Asso[m_t]}$.

\subsection{$(m_1, \dots, m_t)$-sylvester fan and shuffle of associahedra}

First, we describe here the normal fan of the shuffle~$\Asso[m_1] \star \dots \star \Asso[m_t]$.
We refer to \cite{ChapotonPilaud-shuffle} for the general treatment of shuffle of deformed permutahedra.
We fix $m_1 \ge 1, \dots, m_t \ge 1$ with $t > 0$.
Let~$m \eqdef \sum_{1 \le s \le t} m_s$, and~$M_s \eqdef \sum_{1 \le r \le s} m_r$ for~$0 \le s \le t$ (hence, $M_0 = 0$~and~$M_t = m$).

\begin{definition}
The \defn{$(m_1, \dots, m_t)$-sylvester congruence} is the equivalence relation~$\equiv_\mathrm{sylv}^{m_1, \dots, m_t}$ on permutations of~$[m]$ defined by the transitive closure of the rewriting rule~${UjVikW \equiv_\mathrm{sylv}^{m_1, \dots, m_t} UjVkiW}$ for some letters~${i, j, k}$ of~$[m]$ and some words~$U,V,W$ on~$[m]$, such that~$M_{s-1} < i < j < k \le M_s$ for some~$s \in [t]$.
\end{definition}

\begin{definition}
The \defn{$(m_1, \dots, m_t)$-sylvester fan}~$\LodayFan[m_1, \dots, m_t]$ is the fan of~$\R^m$ whose maximal cones are the unions of the regions~$C(\pi)$ of the braid arrangement corresponding to permutations~$\pi$ of~$[m]$ in the same class of the $(m_1, \dots, m_t)$-sylvester congruence.
\end{definition}

The next statement immediately follow from~\cite[Def.~75 \& Prop.~86]{ChapotonPilaud-shuffle} and \cref{prop:Loday}.

\begin{proposition}
The $(m_1, \dots, m_t)$-sylvester fan~$\LodayFan[m_1, \dots, m_t]$ is the normal fan of the shuffle of associahedra
\[
\Asso[m_1] \star \dots \star \Asso[m_t] \eqdef \bigl( \Asso[m_1] \times \dots \times \Asso[m_r] \bigr) + \sum_{\substack{1 \le r < s \le t \\ i \in [m_r], \, j \in [m_s]}} [\b{e}_{M_{r-1} + i}, \b{e}_{M_{s-1}+j}].
\]
\end{proposition}

\subsection{Product slope map}

We now consider $t$ generic linear programs $(\polytope{P}_1, \b c_1), \dots, (\polytope{P}_t, \b c_t)$.
For each~$s \in [t]$, the polytope~$\polytope{P}_s$ is $d_s$-dimensional in~$\R^{d_s}$ and has~$n_s \eqdef m_s+1$ vertices ordered according to $\b c_s \in \R^{d_s}$, and we denote by~$\SlopeMap_s : \R^{d_s} \to \R^{m_s}$ the slope map of~$(\polytope{P}_s, \b c_s)$.

We consider the generic linear program~$(\polytope{P}, \b c)$ where~$\polytope{P} \eqdef \polytope{P}_1 \times \dots \times \polytope{P}_t$ and~$\b c \eqdef (\b c_1, \dots, \b c_t)$.
We let~$d \eqdef \sum_{s \in [t]} d_s$ and~$m \eqdef \sum_{s \in [t]} m_s$.
Note that the number of vertices of~$\polytope{P}$ is $\prod_{s \in [t]} n_s = \prod_{s \in [t]} (m_s+1)$, which is different from~$m+1 = \bigl( \sum_{s \in [t]} m_s \bigr) + 1$.

Throughout, we identify $\prod_{s \in [t]} \R^{d_s}$ with~$\R^d$, and similarly $\prod_{s \in [t]} \R^{m_s}$ with~$\R^m$.
Namely, we have
\(
(\b c_1, \dots, \b c_t) = (c_{1,1}, \dots, c_{1,d_1}, c_{2,1}, \dots, c_{2,d_2}, \dots, c_{t,1}, \dots, c_{t,d_t}).
\)
As in \cref{subsec:slopeMap}, our vertex labeling enables us to consider that~$\c A^{\b\omega} : \prod_{s \in [t]} [n_s] \to \prod_{s \in [r]} [n_s]$ for a generic~$\b\omega \in \R^d$.

\begin{lemma}
\label{lem:formulaACompleteGraph2}
For any generic~$\b\omega \eqdef (\b\omega_1, \dots, \b\omega_t) \in \R^d$, we have
\[
\c A^{\b\omega}(i_1, \dots, i_t) = (i_1, \dots, i_{r-1}, \c A^{\b\omega_r}(i_r), i_{r+1} \dots, i_t),
\]
where~$r \in [t]$ is such that~$\tau^{\b\omega_r}(i_r) = \max\set{\tau^{\b\omega_s}(i_s)}{s \in [t]}$.
\end{lemma}

\begin{proof}
The improving neighbors of~$(i_1, \dots, i_t)$ in~$\polytope{P}_1 \times \dots \times \polytope{P}_t$ are of the form~$(i_1, \dots, j_s, \dots, i_t)$ for some~$s \in [t]$ and some improving neighbor~$j_s$ of~$i_s$ in~$\polytope{P}_s$.
Moreover,
\[
\rho^{\b\omega}\bigl( (i_1, \dots, i_s, \dots i_t), (i_1, \dots, j_s, \dots, i_t) \bigr) = \rho^{\b\omega_s}(i_s, j_s).
\]
Thus, for a fixed~$s \in [t]$, the best improving neighbor~$(i_1, \dots, j_s, \dots, i_t)$ is $(i_1, \dots, \c A^{\b\omega_s}(i_s), \dots, i_t)$, and its slope is~$\tau^{\b\omega_s}(i_s)$.
Hence, the best improving neighbor is~$(i_1, \dots, \c A^{\b\omega_r}(i_r), \dots, i_t)$, for~$r \in [t]$ maximizing~$\tau^{\b\omega_r}(i_r)$.
\end{proof}

\begin{lemma}
\label{lem:projection}
For any generic~$\b\omega, \b\omega' \in \R^d$, if~$\c A^{\b\omega} = \c A^{\b\omega'}$ then~$\c A^{\b\omega_s} = \c A^{\b\omega'_s}$ for all~$s \in [t]$.
\end{lemma}

\begin{proof}
From \cref{lem:formulaACompleteGraph2}, we obtain that
\[
\c A^{\b\omega}(n_1, \dots, n_{s-1},\, i_s,\, n_{s+1}, \dots, n_t) = (n_1, \dots, n_{s-1},\, \c A^{\b\omega_s}(i_s),\, n_{s+1}, \dots, n_t).
\]
Hence, $\c A^{\b\omega}$ indeed determines~$\c A^{\b\omega_s}$ for all~$s \in [t]$.
\end{proof}

\begin{definition}
Define the \defn{product slope map}~$\ProductSlopeMap \!:\! \R^d \to \R^m$ by
\(
\ProductSlopeMap(\b\omega_1, \dots, \b\omega_t) \! \eqdef \! \bigl(\SlopeMap_1(\b\omega_1), \dots, \SlopeMap_t(\b\omega_t)\bigr).
\)
As explained above, this is identified with the concatenation of the vectors~$\SlopeMap_1(\b\omega_1), \dots, \SlopeMap_t(\b\omega_t)$.
\end{definition}

\begin{lemma}
\label{lem:injective2}
On each maximal cone of the pivot fan~$\PivotFan$, the product slope map~$\ProductSlopeMap$ is linear and injective.
\end{lemma}

\begin{proof}
If~$\b\omega$ and~$\b\omega'$ lie in the same maximal cone of~$\PivotFan$, then~$\b\omega_s$ and~$\b\omega'_s$ lie in the same maximal cone of~$\PivotFan[\polytope{P}_s][\b c_s]$ by \Cref{lem:projection}.
Hence, $\ProductSlopeMap$ is linear injective since each~$\SlopeMap_s$ is linear injective.
\end{proof}

For a generic~$\b\omega \in \R^d$, we define~\defn{$\pi^{\b\omega}$} as the permutation of~$[m]$ such that $\ProductSlopeMap(\b\omega)_{\pi^{\b\omega}_1} < \dots < \ProductSlopeMap(\b\omega)_{\pi^{\b\omega}_m}$.

\begin{lemma}
\label{lem:sylvester2}
If the graph of~$\polytope{P}_s$ is complete for each~$s \in [t]$, then for any two generic~$\b\omega, \b\omega' \in \R^d$, we have~$\c A^{\b\omega} = \c A^{\b\omega'}$ if and only if~${\pi^{\b\omega} \equiv_\mathrm{sylv}^{m_1, \dots, m_t} \pi^{\b\omega'}}$.
\end{lemma}

\begin{proof}
We assume that~$\pi^{\b\omega}$ and~$\pi^{\b\omega'}$ are adjacent permutations, the general case follows by induction on the distance from~$\pi^{\b\omega}$ to~$\pi^{\b\omega'}$.
Hence~$\pi^{\b\omega} = XikY$ while~${\pi^{\b\omega'} = XkiY}$ for some letters~${1 \le i < k \le m}$ and some words~$X,Y$ on~$[m]$.
Let~$1 \le r \le s \le t$ be such that~$M_{r-1} < i \le M_r$ and~$M_{s-1} < k \le M_s$, and define~$\underline i \eqdef i - M_{r-1}$ and~$\underline k \eqdef k - M_{s-1}$, so that~$\ProductSlopeMap(\b\omega)_i = \tau^{\b\omega_r}(\underline i)$ and~$\ProductSlopeMap(\b\omega)_k = \tau^{\b\omega_s}(\underline k)$.
We now distinguish three cases.

Assume first that~$r < s$.
As~$\tau^{\b\omega_r}(\underline i) = \ProductSlopeMap(\b\omega)_i$, $\tau^{\b\omega_s}(\underline k) = \ProductSlopeMap(\b\omega)_k$ and~$\tau^{\b\omega_q}(n_q) = -\infty$ for all~$q \in [t]$, the fact that $\ProductSlopeMap(\b\omega)_i < \ProductSlopeMap(\b\omega)_k$ while $\ProductSlopeMap(\b\omega')_i > \ProductSlopeMap(\b\omega')_k$ implies by \cref{lem:formulaACompleteGraph2} that
\begin{align*}
\c A^{\b\omega}(n_1, & \dots, n_{r-1},\, \underline i,\, n_{r+1}, \dots, n_{s-1},\, \underline k,\, n_{s+1}, \dots, n_t) \\
& = (n_1, \dots, n_{r-1},\, \underline i,\, n_{r+1}, \dots, n_{s-1},\, \c A^{\b\omega_s}(\underline k),\, n_{s+1}, \dots, n_t) \\
& \ne (n_1, \dots, n_{r-1},\, \c A^{\b\omega_r}(\underline i),\, n_{r+1}, \dots, n_{s-1},\, \underline k,\, n_{s+1}, \dots, n_t) \\
& = \c A^{\b\omega'}(n_1, \dots, n_{r-1},\, \underline i,\, n_{r+1}, \dots, n_{s-1},\, \underline k,\, n_{s+1}, \dots, n_t).
\end{align*}

Assume now that~$r = s$ and~$\pi^{\b\omega_r} \not\equiv_\mathrm{sylv} \pi^{\b\omega'_r}$.
By the proof of \cref{lem:sylvester1}, we get~${\c A^{\b\omega_r}(\underline i) \ne \c A^{\b\omega'_r}(\underline i)}$.
Hence
\begin{multline*}
\c A^{\b\omega}(n_1, \dots, n_{r-1},\, \underline i,\, n_{r+1}, \dots, n_t)
= (n_1, \dots, n_{r-1},\, \c A^{\b\omega_r}(\underline i),\, n_{r+1}, \dots, n_t) \\
\ne (n_1, \dots, n_{r-1},\, \c A^{\b\omega'_r}(\underline i),\, n_{r+1}, \dots, n_t)
= \c A^{\b\omega}(n_1, \dots, n_{r-1},\, \underline i,\, n_{r+1}, \dots, n_t)
\end{multline*}

Assume finally that~$r = s$ and~$\pi^{\b\omega_r} \equiv_\mathrm{sylv} \pi^{\b\omega'_r}$.
Then for any $(i_1, \dots, i_t)$, the quantities~$\tau^{\b\omega_q}(i_q)$ and~$\tau^{\b\omega'_q}(i_q)$ are maximized for the same~$q \in [t]$.
Thus
\[
\c A^{\b\omega}(i_1, \dots, i_t) = (i_1, \dots, \c A^{\b\omega_q}(i_q), \dots, i_t)
\qquad\text{and}\qquad
\c A^{\b\omega'}(i_1, \dots, i_t) = (i_1, \dots, \c A^{\b\omega'_q}(i_q), \dots, i_t).
\]
Hence, $\c A^{\b\omega} = \c A^{\b\omega'}$ if and only if~$\c A^{\b\omega_q} = \c A^{\b\omega'_q}$.
If~$q \ne r = s$, then~$\c A^{\b\omega_q} = \c A^{\b\omega'_q}$ since $\pi^{\b\omega_q} = \pi^{\b\omega'_q}$.
If~$q = r = s$, then~$\c A^{\b\omega_q} = \c A^{\b\omega'_q}$ by \cref{lem:sylvester1} since $\pi^{\b\omega_r} \equiv_\mathrm{sylv} \pi^{\b\omega'_r}$.

We conclude that
\(
\c A^{\b\omega} = \c A^{\b\omega'}
\;\;\iff\;\;
r = s \text{ and } \pi^{\b\omega_r} \equiv_\mathrm{sylv} \pi^{\b\omega'_r}
\;\;\iff\;\;
\pi^{\b\omega} \equiv_\mathrm{sylv}^{m_1, \dots, m_t} \pi^{\b\omega'}
\).
\end{proof}

\subsection{Pivot fan of a product of simplices}

We are now ready to show that the product slope map sends the pivot fan of~$\simplex_{m_1} \times \dots \times \simplex_{m_t}$ to the $(m_1, \dots, m_t)$-sylvester fan~$\LodayFan[m_1, \dots, m_t]$, which refines \cref{conj:main}.

\begin{theorem}
\label{thm:PivotFanProductSimplices}
For any full dimensional simplices $\simplex_1 \subset \R^{m_1}, \dots, \simplex_t \subset \R^{m_t}$ and any generic direction~$\b c \in \R^{m_1 + \dots + m_t}$, the product slope map $\ProductSlopeMap$ is a piecewise linear isomorphism from the pivot fan~$\PivotFan[\simplex_1\times\dots\times\simplex_t]$ to the $(m_1, \dots, m_t)$-sylvester fan~$\LodayFan[m_1, \dots, m_t]$.
\end{theorem}

\begin{proof}
Same argument as in the proof of \cref{thm:PivotFanSimplex}, using \cref{lem:injective2} instead of \cref{lem:injective1} and \cref{lem:sylvester2} instead of~\cref{lem:sylvester1}.
\end{proof}

\begin{corollary}
For any full dimensional simplices~$\simplex_1 \subset \R^{m_1}, \dots, \simplex_t \subset \R^{m_t}$ and any generic direction~$\b c \in \R^{m_1 + \dots + m_t}$, the product slope map~$\ProductSlopeMap$ sends the pivot fan of~$\simplex_1\times\dots\times\simplex_t$ to the normal fan of
\begin{itemize}
\item the $t$-permutahedron when~$m_1 = \dots = m_t = 1$,
\item the $m_1$-associahedron when~$t = 1$,
\item the $(m_1,m_2)$-constrainahedron when~$t = 2$, see~\cite{BottmanPoliakova},
\item the $(t-1,m_t)$-multiplihedron when~$m_1 = \dots = m_{t-1} = 1$, see \cite[Sect.~3.2]{ChapotonPilaud-shuffle}.
\end{itemize}
\end{corollary}

Note that the first three are proved in~\cite{BlackLutjeharmsSanyal}, while only the vertex count of the last one is proved in~\cite{BlackLutjeharmsSanyal}.

\begin{remark}
Similar to \cref{rem:bijection}, the product slope map induces a bijection from the arborescences~$\set{\c A^{\b\omega}}{\b\omega \in \R^m}$ on~$(\simplex_1 \times \dots \times \simplex_t, \b c)$ to the $(m_1, \dots, m_t)$-cotrees in the sense of~\cite[Sect.~4]{ChapotonPilaud-shuffle} (note that \cite[Sect.~4]{ChapotonPilaud-shuffle} only presents $(m,n)$-cotrees, but the definition extends straightforward to tuples).
\end{remark}

\section{Acknowledgments}

We are grateful to Raman Sanyal for introducing us the fascinating world of pivot polytopes, for sharing his results on pivot polytopes of simplices, for enthusiastically embracing the conjecture on pivot polytopes of products of simplices, and for encouraging us to write this paper.
We also thank him for inviting Germain Poullot in Frankfurt for a research visit on pivot polytopes.
We thank Alexander Black and Raman Sanyal for comments on a preliminary version of this paper.

\cref{conj:main} arose during the workshop ``Combinatorics and Geometry of Convex Polyhedra'' held at the Simons Center for Geometry and Physics in March 2023.
We are grateful to the organizers (Karim Adiprasito, Alexey Glazyrin, Isabella Novik, and Igor Pak) for this inspiring event, and to all participants for the wonderful atmosphere.

Finally, we are grateful to two anonymous referees for various suggestions on the presentation of this paper.

\bibliographystyle{alpha}
\bibliography{pppssa}
\label{sec:biblio}

\end{document}